\documentclass[11pt,twoside,a4paper,reqno]{amsart}

\newtheorem{thm}{Theorem}[section]
\newtheorem{lem}[thm]{Lemma}
\newtheorem{prop}[thm]{Proposition}

\theoremstyle{definition}

\theoremstyle{remark}
\newtheorem*{remark}{Remark}

\theoremstyle{plain}

\numberwithin{equation}{section}

\newcommand{\R}{{\mathbb{R}}}

\DeclareMathAlphabet{\euls}{U}{eus}{m}{n}

\DeclareMathAlphabet{\eulr}{U}{eur}{m}{n}
\DeclareMathAlphabet{\eulrb}{U}{eur}{b}{n}

\def\preisomto{\vbox{\hbox to
               14pt{\hfill$\sim$\hfill}\nointerlineskip\vskip -0.3pt
               \hbox to 14pt{\rightarrowfill}}}

\def\prelongisomto{\vbox{\hbox to
                17pt{\hfill$\sim$\hfill}\nointerlineskip\vskip -0.3pt
                \hbox to 17pt{\rightarrowfill}}}

\newcommand{\doublelongrightarrow}{\longrightarrow \kern-14pt
\longrightarrow}

\def\trait{\hbox to 4mm{\hrulefill}}

\DeclareMathSymbol{\Ima}{\mathord}{symbols}{"3D}

                       {\end{itemize}}

                       {\end{enumerate}}

\makeatletter
 {\begin{list}{{$\boldsymbol{\cdot}$}}%
          {\let\@trivlist\@trivlist
             \settowidth{\labelwidth}{{$\boldsymbol{\cdot}$}}%
             \setlength{\leftmargin}{10pt}%
             \addtolength{\leftmargin}{\labelsep}%
             \ifnum\@listdepth=0
               \setlength{\itemindent}{\parindent}%
             \else
               \addtolength{\leftmargin}{\parindent}%
             \fi
             \setlength{\itemsep}{\z@}%
             \setlength{\parsep}{\z@}%
             \setlength{\topsep}{\z@}%
             \setlength{\partopsep}{\z@}%
             \addtolength{\topsep}{-\parskip}%
             \addtolength{\partopsep}{\parskip}%
          }%
        }%
        {\end{list}}%
\makeatother

\usepackage{xcolor}

\begin{document}

\centerline{\LARGE \bf Existence of pulses for a reaction-diffusion}

\medskip

\centerline{\LARGE \bf system of blood coagulation}

\vspace{1cm}

\centerline{\bf  N. Ratto$^1$, M. Marion$^1$, V. Volpert$^{2,3,4,5}$}

\vspace{0.5cm}

\centerline{$^1$Institut Camille Jordan, UMR 5585 CNRS, Ecole Centrale de Lyon}

\centerline{69134 Ecully, France}

\centerline{$^2$Institut Camille Jordan, UMR 5585 CNRS, University Lyon 1}

\centerline{69622 Villeurbanne, France}

\centerline{$^3$
INRIA, Universit\'e de Lyon, Universit\'e Lyon 1, Institut Camille Jordan}
\centerline{
43 Bd. du 11 Novembre
1918, 69200 Villeurbanne Cedex, France}

\centerline{$^4$
Peoples’ Friendship University of Russia (RUDN University)}
\centerline{
6 Miklukho-Maklaya St, Moscow, 117198, Russian Federation}

\centerline{$^5$
Poncelet Center, UMI 2615 CNRS, 11 Bolshoy Vlasyevskiy, 119002 Moscow}
\centerline{
Russian Federation}

\vspace{1cm}

\noindent
{\bf Abstract.}
The paper is devoted to the investigation of a reaction-diffusion system of equations describing the process
of blood coagulation. Existence of pulses solutions, that is, positive stationary solutions with zero limit
at infinity is studied. It is shown that such solutions exist if and only if the speed of the travelling wave
described by the same system is positive. The proof is based on the Leray-Schauder method using topological
degree for elliptic problems in unbounded domains and a priori estimates of solutions in some
appropriate weighted spaces.

\vspace{1cm}

\noindent
{\bf Key words:} reaction-diffusion system, blood coagulation, existence of pulses, Leray-Schauder method

\medskip

\noindent
{\bf AMS subject classification:} 35K57

\vspace*{1cm}

\section{Introduction}

{Hemostasis} is a physiological process which aims to prevent bleeding in the case of blood vessel damage. It includes vasoconstriction, platelet plug formation and blood coagulation in plasma with the formation of fibrin clot.   In this work we will focus on the blood coagulation process. A malfunction in this process can lead to thrombosis or to various bleeding disorders.
The process of blood coagulation has three main stages: initiation, amplification and clot growth arrest. They are determined by chemical reactions in plasma between different proteins (blood factors), among which the most important role is played by thrombin. Thrombin is an enzyme catalyzing the conversion of the fibrinogen into the fibrin polymer which forms the clot.

In the process of blood coagulation initiated by the vessel wall damage (extrinsic pathway), the initial quantity of thrombin is produced due to the interaction of the tissue factor (TF) with the activated factors VIIa. During the amplification phase, there is a positive feedback loop of thrombin production through the activation of the factors V, VIII, IX, X and XI. It can be noted that hemophilia is characterized by the lack of factor VIII, IX or XI.  Finally, clot growth is stopped by the reaction of antithrombin with thrombin, by the protein C pathway, and due to the blood flow removing blood factors from the clot.

The amplification phase of blood coagulation can start only if the amount of thrombin produced during the initiation phase exceeds certain threshold. In this work we will show that this critical thrombin concentration is determined by a particular solution (pulse) of the reaction-diffusion system of equations describing the coagulation cascade. On the other hand, the amplification phase can be described as a reaction-diffusion wave \cite{A2}, \cite{pogorelova_influence_2014}, \cite{tokarev_volpert2006}, \cite{zarnitsina_dynamics_2001}.
The main result of this work affirms that the pulse solution exists if and only if the wave speed is positive. Thus, we obtain two conditions of blood coagulation: the wave speed should be positive providing the existence of the pulse solution; the distributions of the concentrations of blood factors after the initiation phase should be greater than those for the pulse solution.

We consider the reaction-diffusion system  of equations \cite{BOUCHNITA201674}:

\begin{equation}
\frac{\partial \mathbf{v}}{\partial t}=D {\frac{\partial^2 \mathbf{v}}{\partial x^2}}+\mathbf{F}(\mathbf{v}),
\label{Eq generale}
\end{equation}
%
%
where  $\mathbf{v}=(v_1,...,v_8)$ and $\mathbf{F} = (F_1,...,F_8)$,
\begin{equation}
\mathbf{F}(\mathbf{v})=
\begin{cases}
F_1(\mathbf{v})=k_1v_3v_6-h_1v_1 \\
F_2(\mathbf{v})=k_2v_4v_5-h_2v_2 \\
F_3(\mathbf{v})=k_3v_8(\rho_3-v_3)-h_3v_3 \\
F_4(\mathbf{v})=k_4v_8(\rho_4-v_4)-h_4v_4 \\
F_5(\mathbf{v})=k_5v_7(\rho_5-v_5)-h_5v_5 \\
F_6(\mathbf{v})=(k_6v_5+\overline{k_6}v_2)(\rho_6-v_6)-h_6v_6\\
F_7(\mathbf{v})=k_7v_8(\rho_7-v_7)-h_7v_7\\
F_8(\mathbf{v})=(k_8v_6+ \overline{k_8}v_1)(\rho_8-v_8)-h_8v_8 \\
\end{cases}.
\label{Def F}
\end{equation}
Parameters $(k_i)_i$, $(h_i)_i$ and $(\overline{k_i})_i$ and $(\rho_i)_i$ are positive constants. {The matrix $D=diag(D_i)$ is a diagonal matrix with positive diagonal elements $D_i$.}

In this system, $v_3, v_4, v_5, v_6$ and $v_7$ denote, respectively, the concentrations of the activated Factors $Va, VIIIa, IXa, Xa$ and $XIa$, $v_8$ is the concentration of activated Factor $IIa$ (thrombin), $v_1$ and $v_2$ are the concentrations of prothrombinase and intrinsic tenase complexes. The constants $k_i$ and $\bar{k_i}$ are the activation rates of the corresponding factors by other factors or complexes, while the constants $h_i$ are the rates their inhibition. {The constants $D_i$ are the diffusion coefficients of each factor.} The thrombin concentration, $v_8$, has a major role in the coagulation process. It will also have a particular importance in the mathematical study. We will use notation $T=v_8$ and $T^0=\rho_8$.

Let us introduce the set $\mathcal{C}$:
\begin{equation}
\mathcal{C}=\{\mathbf{v}=(v_1,...,v_8)\in\R^8_+ | v_i\leq \rho_i \text{ for } i \in \{3,...,8\} \}.
\label{Def C}
\end{equation}
Then the system \eqref{Eq generale} is a monotone system on $\mathcal{C}$ , that is:
\begin{equation}
\forall \mathbf{v}\in \mathcal{C}, \quad \frac{\partial F_i}{\partial v_j}(\mathbf{v})\geq 0, \quad \forall i\ne j.
\label{F monotone}
\end{equation}

Note that $\mathbf{v}=\mathbf{0}$ is a zero of $\mathbf{F}$.
In order to determine other zeros, let us express $v_i$ through $T$ from the equations $F_i(\mathbf{v})=0$, $i=1,...,7$:
$v_i = \phi_i(T)$ (see Appendix A for the explicit form of these functions). Substituting them in the equation $F_8(\mathbf{v})=0$,
we obtain:
\begin{equation}
P(T) \equiv (k_8\phi_6(T)+\bar{k_8}\phi_1(T))(T^0-T)-h_8T = 0 . \label{Equation P}
\end{equation}
It can be directly verified that  {  $P(T)$ is a rational fraction that takes the form $P(T) = \frac {T Q(T)} {S(T)}$}, where $Q(T)$ and  $S(T)$ are third-order polynomials. { Moreover 
$Q(T)=aT^3+bT^2+cT+d$
 with }  $a<0$, and $S(T) > 0$ for $T \geq 0$ { (the reader is referred to the appendix)}.

{ Hereafter, we will assume that $P$ satisfies  the following properties}:
\begin{equation}
\begin{cases}
{P \text{ possesses exactly three non negative zeros: }T^+=0 <\bar{T}<T^-} \\
{P'(T^+)<0, \; P'(\bar{T})>0, \; P'(T^-)<0.}
\end{cases}
\label{Condition P}
\end{equation}
{In view of the form of the rational fraction $P(T)$, the above conditions mean that the third order polynomial $Q$ with negative leading coefficient $a$ has exactly two positive roots. Also it is easy to check that the conditions on the derivatives in \eqref{Condition P} also read 
\begin{equation}
Q(0) = d <0, \;Q'(\bar{T})>0, \;Q'(T^-)<0.
\label{Condition Q}
\end{equation}
In particular ${Q}(T) > 0$ for some $T > 0$}. This assumption is biologically justified, and we will return to it in the discussion (Section 5).

The non negative roots $T^+$, $\bar{T}$ and $T^-$ correspond to three zeros of $\mathbf{F}$: $\mathbf{w}^+=(\phi(T^+),T^+)=\mathbf{0} $, $\bar{\mathbf{w}}=(\phi(\bar{T}),\bar{T})$ and $\mathbf{w}^-=(\phi(T^-),T^-)$.
Moreover we have $\mathbf{w}^+=\mathbf{0}<\bar{\mathbf{w}}<\mathbf{w}^-$, where the inequalities between the  vectors are understood
component-wise.


As stated before, thrombin propagation in blood plasma is described by a traveling wave solution of the system \eqref{Eq generale}. A wave solution of \eqref{Eq generale} is a solution that can be written as $\mathbf{v}(x,t)=\mathbf{u}(x-ct)$ where the wave speed $c\in \R$ is unknown. Hence we look for a function $\mathbf{u}$ and a constant $c$ that are solutions of the problem:
\begin{equation}
 \begin{cases}
D\mathbf{u}''+c\mathbf{u}'+\mathbf{F}(\mathbf{u})=0, \\
\mathbf{u}(\pm \infty)=\mathbf{w}^{\pm}.
 \end{cases}
\label{Eq wave}
 \end{equation}
Under the conditions \eqref{F monotone} and \eqref{Condition P}, the problem \eqref{Eq wave} possesses a unique solution (up to some translation in space for $u$). This solution is a monotonically decreasing vector-function. This results are presented in \cite{TravelingWave}.

Biologically, it has been noted that the amplification of thrombin generation occurs if the amount of thrombin produced during the initiation phase reaches a certain threshold. We will show that this threshold is a pulse of the stationary system, that is a function $\mathbf{w}: \R \rightarrow \R^8$ that satisfies the following problem:
\begin{equation}
\begin{cases}
D\mathbf{w}''+\mathbf{F}(\mathbf{w})=0, \\
\mathbf{w}'(0)= 0 , \; \mathbf{w}(+\infty)=0, \\
\mathbf{w}'<0 \quad \text{ for } x>0.
\end{cases}
\label{Eq stat}
\end{equation}

The link between the solutions of the wave problem \eqref{Eq wave} and the pulse problem \eqref{Eq stat} is given by the main result of this work:
\begin{thm}
Under the condition \eqref{Condition P} the problem \eqref{Eq stat} has a solution if and only if the wave speed $c$ in the problem \eqref{Eq wave} is positive.
\label{Thm principal}
\end{thm}


The proof of the theorem mainly relies on the Leray-Schauder method and some homotopy arguments.  Therefore, we will introduce some appropriate homotopy deformation in Section \ref{Sec Homotopy}. 
In Section \ref{Sec Estimation} we  obtain a priori estimates of solutions in some weighed H\"older spaces using positivity of the wave speed. These estimates are independent of the parameter of the homotopy. Hence, the value of the  topological degree is preserved along the homotopy providing the existence of solutions. We finish the proof of the existence of solution of problem \eqref{Eq stat}
 in Section \ref{Sec Proof}. Also in this section we will show that the problem has no solution if $c\leq 0$.


\section{Homotopy}
\label{Sec Homotopy}

In order to prove the existence part of  Theorem 1.1 
we introduce in this section a homotopy deformation and highlight some its properties.

\subsection{Description of the homotopy}
The homotopy aims to modify continuously the function $\mathbf{F}$. We will only modify the last component $F_8(\mathbf{v})=(k_8v_6+ \overline{k_8}v_1)(T^0-T)-h_8T$ of \eqref{Def F} which depends on $T, v_6$ and $v_1$ into a new function depending only on $T$. Hence, the last equation will be independent from the other equations. In Lemma \ref{lem exit T} (Section \ref{Sec Proof}) we will show that this equation possesses a unique solution. For notation purpose the initial function, $\mathbf{F}$, corresponds to $\tau=0$ and is written $\mathbf{F}^{0}$. The homotopy functions reads $\mathbf{F}^{\tau}$, $\tau \in [0,1]$. The homotopy is defined in two steps, and we introduce some $\tau_1$ in $(0,1)$.

For the first step we introduce a smooth function $g$ that will be chosen below. We will construct the homotopy in such a way that the zeros of the function $\mathbf{F}^{\tau}$ do not change and coincide with the zeros of $\mathbf{F}$. Hence, assuming that $P$ satisfies the conditions \eqref{Condition P}, we impose that $g$ satisfies the following condition:

\begin{equation}
g(s)\geq 0 \text{ for } s\geq 0, \text{ Supp } g \subset (\bar{T},T^-).
\label{condition g}
\end{equation}
Next, for $\tau \in (0,\tau_1)$ we define the homotopy by the equality:
\begin{equation}
F_8^{\tau}(\mathbf{v})=(k_8v_6+ \overline{k_8}v_1)(T^0-T)-h_8T+ \tau g(T).
\label{Def Hom 1}
\end{equation}

 At the second step of the homotopy we will deal with the variables $v_1$ and $v_6$.
 We will replace them by the functions $\phi_1(T)$ and $\phi_6(T)$ given by \eqref{Equation phi1} and \eqref{Equation phi6}
 (see also \eqref{Equation P}) without modifying the zeros of $\mathbf{F}^{\tau}$.
Hence, for $\tau \in (\tau_1,1)$ we set:
\begin{equation}
F_8^{\tau}(\mathbf{v})=(k_8(\alpha^{\tau}v_6+\beta^{\tau}\phi_6(T))+ \overline{k_8}(\alpha^{\tau}v_1+\beta^{\tau}\phi_1(T))(T^0-T)-h_8T+ \tau_1 g(T),
\label{Def Hom 2}
\end{equation}
where
\begin{equation}
\alpha^{\tau}=\frac{1-\tau}{1-\tau_1} \; , \quad
\beta^{\tau}=\frac{\tau-\tau_1}{1-\tau_1} \; .
\end{equation}
It will be convenient to introduce notation:

\begin{equation}
\begin{array}{c c c}

\begin{cases}
\alpha^{\tau}=\frac{1-\tau}{1-\tau_1} \text{ if } \tau\geq \tau_1\\
\alpha^{\tau}=1 \text{ if } \tau<\tau_1
\end{cases}
, &
\begin{cases}
\beta^{\tau}=\frac{\tau-\tau_1}{1-\tau_1} \text{ if } \tau\geq \tau_1\\
\beta^{\tau}=0  \text{ if } \tau<\tau_1
\end{cases}
, &
\begin{cases}
\gamma^{\tau}=\tau_1 \text{ if } \tau\geq \tau_1\\
\gamma^{\tau}=\tau \text{ if } \tau<\tau_1\\
\end{cases}
.
\end{array}
\end{equation}
Then equalities \eqref{Def Hom 1} and \eqref{Def Hom 2} can be put together as follows:

\begin{equation}
F_8^{\tau}(\mathbf{v}) = \alpha^{\tau}F_8(\mathbf{v})+\beta^{\tau}P(T)+\gamma^{\tau}g(T).
\label{Def Hom 3}
\end{equation}

\subsection{Preservation of some properties}

Along the homotopy, and independently of the choice of $\tau_1$,  some properties of $\mathbf{F}$ are preserved: the monotony, the stationary points and their stability.
For the monotony we have the following result.

\begin{lem}
For all $\tau$ in $[0,1]$, the function $\mathbf{F}^{\tau}$ satisfies the monotony property on $\mathcal{C}$:
\begin{equation}
\forall \mathbf{v}\in \mathcal{C}, \quad \frac{\partial F^{\tau}_i}{\partial v_j}(\mathbf{v})\geq 0, \quad \forall i\ne j .
\label{F tau monotone}
\end{equation}
\end{lem}

\begin{proof}
Since only the last component of $\mathbf{F}$ differs during the homotopy, we only need to verify the result for $F_8^{\tau}$.
Using the expression \eqref{Def Hom 3} and  the monotony of $\mathbf{F}$ given by \eqref{F monotone} we have:
$$\frac{\partial F^{\tau}_8}{\partial v_j}(\mathbf{v})= \alpha^{\tau}\frac{\partial F_8}{\partial v_j}(\mathbf{v}) \geq 0, \quad \text{ for } j \in\{1,...,7\}.$$
\end{proof}

We will now show that the homotopy does not have any impact on the zeros of $\mathbf{F}^{\tau}$.

\begin{lem}
For all $\tau$ in $[0,1]$, the zeros of $\mathbf{F}^{\tau}$ are exactly the same as the zeros of $\mathbf{F}$.
\end{lem}

\begin{proof}
 As before, since the homotopy only modifies the last component of $\mathbf{F}^{\tau}$, we only need to prove that the zeros of $F_8^{\tau}$ remain unchanged as $\tau$ varies.
In order to find the zeros of $\mathbf{F}^{\tau}$ we use the same method as for the zeros of $\mathbf{F}$ in the Appendix \ref{App Exp phi}. Let  $\mathbf{v}=(v_1,...,v_7,T)$ denote some zero of $\mathbf{F}^{\tau}$. Since $F_i^{\tau}=F_i$ for $1\leq i\leq 7$, the relations \eqref{Equation phi1}-\eqref{Equation phi7} remain unchanged.  Also, in view of \eqref{Equation P} and \eqref{Def Hom 3} we have:
\begin{equation}
F_8^{\tau}(\boldsymbol{\phi}(T),T)=\alpha^{\tau}F_8(\boldsymbol{\phi}(T),T)+\beta^{\tau}P(T)+\gamma^{\tau}g(T)=P(T) + \gamma^{\tau}g(T)=P^{\tau}(T).
\label{Equation P tau}
\end{equation}
Hence the zeros of $\mathbf{F}^{\tau}$ are given by:
\begin{equation}
\mathbf{F}^{\tau}(\mathbf{v})=\mathbf{0} \iff \mathbf{v}=(\phi(T),T) \text{ and } P^{\tau}(T)=0.
\label{Zero F tau}
\end{equation}
Since the zeros of $P^{\tau}$ and $P$ coincide, we have $P^{\tau}(T)=0 \iff P(T)=0$. Hence $\mathbf{F}^{\tau}(\mathbf{w})=\mathbf{0}  \iff \mathbf{F}(\mathbf{w})=\mathbf{0} $.
\end{proof}

Therefore the zeros of $\mathbf{F}^{\tau}$ are $\mathbf{w}^+, \bar{\mathbf{w}}$ and $\mathbf{w}^-$.
Let us now investigate their stability.
Let $\mathbf{w}^*$ refer to either one them.
For this purpose we will need the following lemma.

\begin{lem}\label{lem F_i(phi,T)=0}
For all $T\geq 0$ and $i=1,...,7$ we have the equalities:
\begin{equation}
F_i(\boldsymbol{\phi}(T),T)=0
\end{equation}
and
\begin{equation}\label{Eq dFi/dT}
\frac{d F_i}{d T}(\phi(T),T)=\sum_{j=1}^7\frac{\partial F_i}{\partial v_j}(\phi(T),T)\phi_j'(T) +\frac{\partial F_i}{\partial v_8}(\phi(T),T)=0.
\end{equation}
\end{lem}
The first equality follows from the definition of the functions $\phi_i$, $1\leq i \leq 7$ given in the Appendix \ref{App Exp phi}. The second one is obtained by differentiating  the first one.

The next result concerns the stability of these zeros.
\begin{prop}\label{lem Stability}
The sign of the principal eigenvalue of the Jacobian $(\mathbf{F}^{\tau})'(\mathbf{w}^*)$ does not depend on $\tau \in [0,1]$.
\end{prop}

\begin{proof}
Since the function $\mathbf{F}^{\tau}$ is monotone on $\mathcal{C}$ for every $\tau\in[0,1]$, then the Perron-Frobenius theorem guarantees that the principal eigenvalue, that is, the eigenvalue with maximal real part  of the Jacobian matrix is real.

Let $\mathbf{w}^*=(w_1^*,...,w_7^*,T^*)$, $T^*\geq 0$ be one of the zeros of $\mathbf{F}^{\tau}$.
The stability of $\mathbf{w}^*$ is preserved during the homotopy if the principal eigenvalue of the Jacobian matrix does not change sign as $\tau$ varies. Let us denote by $\mathcal{J}^{\tau}=(\mathbf{F}^{\tau})'$ the Jacobian matrix of $\mathbf{F}^{\tau}$.  To prove that the principal eigenvalue does not change sign as $\tau$ varies, we will show that the determinant of $\mathcal{J}^{\tau}(\mathbf{w}^*)$ is different from zero for all values of $\tau \in [0,1]$. To this end, let us check that  $\text{Ker}\mathcal{J}^{\tau}(\mathbf{w}^*)=\{ \mathbf{0} \}$.

Let $\mathbf{q} \in \text{Ker}\mathcal{J}^{\tau}(\mathbf{w}^*)$. Then we have:
\begin{equation}\label{Ji q}
(\mathcal{J}^{\tau}(\mathbf{w}^*)\mathbf{q})_i=\sum_{j=1}^7\frac{\partial F_i}{\partial v_j}(\mathbf{w}^*)q_j +\frac{\partial F_i}{\partial v_8}(\mathbf{w}^*)q_8=0.
\end{equation}
For $i=3$ we get:

\begin{equation}\label{J3 q}
\frac{\partial F_3}{\partial v_3}(\mathbf{w}^*)q_3 +\frac{\partial F_3}{\partial T}(\mathbf{w}^*)q_8 =0 .
\end{equation}
Hence,

$$ q_3=- {\frac{\partial F_3}{\partial v_3}(\mathbf{w}^*)}/{\frac{\partial F_3}{\partial T}(\mathbf{w}^*)}q_8. $$
From Lemma \ref{lem F_i(phi,T)=0} applied for $i=3$ it follows that

\begin{equation}
\frac{d F_3}{d T}(\mathbf{w}^*)=\frac{\partial F_3}{\partial v_3}(\mathbf{w}^*)\phi_3'(T) +\frac{\partial F_3}{\partial v_8}(\mathbf{w}^*).
\label{Eq dF_3/dT}
\end{equation}
Hence, applying \eqref{Eq dF_3/dT} and \eqref{J3 q} we have:

$$ (\mathcal{J}^{\tau}(\mathbf{w}^*)\mathbf{q})_3 = 0 \iff q_3=\phi_3'(T^*)q_8. $$
We proceed with the other components in the following order: $i=4,7,5,2,$ $6, 1$. This successively provides:
 $q_i=\phi_i'(T^*)q_8 , \;\; i=1,...,7.$
Then the last component of $\mathcal{J}^{\tau}(\mathbf{w}^*).q$ is given by the equality:

$$ (\mathcal{J}_8^{\tau}(\mathbf{w}^*)\mathbf{q})_8 = \frac{\partial F_8^{\tau}}{\partial v_8}(\mathbf{w}^*)q_8+\frac{\partial F_8^{\tau}}{\partial v_1}(\mathbf{w}^*)q_1+\frac{\partial F_8^{\tau}}{\partial v_6}(\mathbf{w}^*)q_6 = $$

$$ (\frac{\partial F_8^{\tau}}{\partial v_8}(\mathbf{w}^*)+\frac{\partial F_8^{\tau}}{\partial v_1}(\mathbf{w}^*)\phi_1'(T^*)+\frac{\partial F_8^{\tau}}{\partial v_6}(\mathbf{w}^*)\phi_6'(T^*))q_8 = P'(T^*)q_8. $$
By virtue of \eqref{Condition P}, $P'(T^*)\ne 0$. Hence $q_8=0$,  implying $\mathbf{q}=\mathbf{0}$ and concluding the proof.
\end{proof}

Proposition \ref{lem Stability} affirms that stability of the stationary points $\mathbf{w}^-$, $\bar{\mathbf{w}}$ and $\mathbf{w}^+$ does not depend on $\tau$. Hence their stability is same for $\mathbf{F}$ and $\mathbf{F}^{\tau}$ for all $\tau \in [0,1]$. Consider $\tau=1$.   The Jacobian matrix $\mathcal{J}^1$ can be reduced to a triangular matrix by taking the component in the order $i=8,3,4,7,5,2,6, 1$. Then the eigenvalue of matrix $\mathcal{J}^1(\mathbf{w}^*)$ are given by its diagonal elements:
$$ P'(T^*),\ -h_1,\ -h_2,\ -k_3T^*-h_3,\ -k_4T^*-h_4,\ -h_5,\ -h_6,\ -k_7T^*-h_7, $$
where $T^*$ refers to either $T^+$, $\bar{T}$ or $T^-$. The last seven eigenvalues are negative for any of the points $\mathbf{w}^-$, $\bar{\mathbf{w}}$ and $\mathbf{w}^+$. Consequently, their stability depends only on the sign of $P'(T^*)$: if $P'(T^*)<0$ the point is stable, if $P'(T^*)>0$ the point is unstable. Thus, the points $\mathbf{w^-}$ and $\mathbf{w^+}$ are stable, while $\bar{\mathbf{w}}$ is unstable.

\section{Functional spaces and a priori estimates}
\label{Sec Estimation}

In this section we will introduce functional spaces and will obtain a priori estimates of solutions.

\subsection{H\"older spaces}

We introduce H\"older space $\mathcal{C}^{k+\alpha}(\R_+)$, $\alpha \in (0,1)$ consisting of vector-functions from $\mathcal{C}^k$ bounded on $\R_+$ together with their derivatives up to the order $k$, and the derivative of order $k$ satisfies  H\"older condition. This space is equipped with the usual H\"older norm. We set:
 \begin{equation}
\begin{cases}
E^1 = \{ \mathbf{w} \in \mathcal{C}^{2+\alpha}(\R_+), \mathbf{w}'(0)=0 \}, \\
E^2=\mathcal{C}^{\alpha}(\R_+).
\end{cases}
\end{equation}
We now introduce the weighted spaces $E^1_{\mu}$ and $E^2_{\mu}$ where $\mu$ is the weight function, $\mu(x)=\sqrt{1+x^2}$. The norm in these spaces is defined by the equality:

\begin{equation}
\|\mathbf{w}\|_{E^i_{\mu}}=\|\mathbf{w}\mu\|_{E^i}, \quad i=1,2.
\end{equation}
Thus, we consider the operator $A^{\tau}:  E^1_{\mu} \to E^2_{\mu}$,

\begin{equation}\label{Def A tau}
 A^\tau(\mathbf{w}) = D\mathbf{w}''+\mathbf{F}^{\tau}(\mathbf{w}).
\end{equation}
We are looking for positive monotone solutions of the equation $A^{\tau}(\mathbf{w})=\mathbf{0}$ such that $\mathbf{w}\in E^1_{\mu}$.



\subsection{Bounded solutions}

We will obtain a priori estimates of solutions of the  equation

\begin{equation}\label{Eq stat tau}
D\mathbf{w}''+\mathbf{F}^{\tau}(\mathbf{w})=\mathbf{0}
\end{equation}
with the {boundary} conditions:

\begin{equation}\label{Eq stat tau boundaries}
\mathbf{w}'(0)=\mathbf{w}(+\infty)=\mathbf{0}
\end{equation}
assuming that

\begin{equation}\label{Eq stat tau mono}
\mathbf{w}'<\mathbf{0} \quad \text{ for } x>0.
\end{equation}
Clearly, a solution of \eqref{Eq stat tau}-\eqref{Eq stat tau boundaries} with  condition \eqref{Eq stat tau mono} is positive.
The corresponding wave problem  \eqref{Eq wave} becomes as follows:

\begin{align}
D\mathbf{u}''+c^{\tau}\mathbf{u}'+\mathbf{F}^{\tau}(\mathbf{u})=0, \label{Eq wave tau equation}\\
\mathbf{u}(\pm \infty)=\mathbf{w}^{\pm}\label{Eq wave tau condition}.
\end{align}

We begin with $L^{\infty}$ of solutions of problem  \eqref{Eq stat tau}, \eqref{Eq stat tau boundaries}. We first prove the following lemma.

\begin{lem}\label{lem Psi}
There exists a function $\mathbf{\Psi}(s)=(\psi_1,...,\psi_8)$ defined for $s\geq 0$ such that:
\begin{equation}
\begin{cases}
\mathbf{\Psi}(0)=\mathbf{w}^-, \mathbf{\Psi}>\mathbf{w}^- \quad \text{ for } s>0, \\
\psi_i\rightarrow w_{i,\ell} \text{ as } s\rightarrow \infty , \;\; i=1,...,7 ,\\
\mathbf{F}^{\tau}(\mathbf{\Psi})<0 \quad \text{ for } s>0 \text{ and } \tau\in[0,1],
\end{cases}
\end{equation}
where $w_{i,\ell}=\underset{s\rightarrow +\infty}{\text{lim}} \phi_i(s)$ for $1\leq i \leq 7$.
\end{lem}

\begin{proof}
We set $\psi^0_i(s)=\phi_i(T^-+s)$, $i=1,...,7$,  $\psi^0_8(s)=T^-+s$, and $\mathbf{\Psi}^0(s)=(\psi^0_1(s),...,\psi^0_8(s))$.
Then we have $\mathbf{\Psi}^0(0)=w^-$. Since the functions $\phi_i$ are increasing, then $\mathbf{\Psi}^0(s)>\mathbf{w}^-$ for $s>0$.
Finally,  $\underset{s\rightarrow \infty}{\text{lim}} \psi^0_i(s)=w_{i,\ell}$ by definition of these functions.

We note that $F^{\tau}_i(\mathbf{\Psi}^0(s))=0$, $i=1,...,7$, 
 $F_8^{\tau}(\mathbf{\Psi}^0(s))=P^{\tau}(T^-+s)<0$ for $s > 0$.
Next, set $\psi_i(s)=\phi_i(T^-+\kappa_i s)$, where $\kappa_i$ are some numbers,
$\mathbf{\Psi}(s)=(\psi_1(s),...,\psi_8(s))$.
It can be directly verified that $F^{\tau}_i(\mathbf{\Psi}(s))<0$, $i=1,...7$ for $s > 0$ if the following conditions are satisfied:

\begin{equation}\label{Cond kappa}
  \kappa_1>\kappa_6>\kappa_2>\kappa_5>\kappa_3>\kappa_4>\kappa_7> \kappa_8 .
\end{equation}
%
%
%
%
%

Consider now $F^{\tau}_8(\mathbf{\Psi}(s))$.
Let $\kappa_i = 1 + \varepsilon \lambda_i$, where $\lambda_i>0$ are chosen to satisfy \eqref{Cond kappa}.
Since $F_8^{\tau}(\mathbf{\Psi}^0(s))=P^{\tau}(T^-+s)<0$ for $s > 0$, then
$F^{\tau}_8(\mathbf{\Psi}(s)) < 0$ for $s \geq s_0 > 0$ with any $s_0$ and all $\varepsilon$ sufficiently small.
It remains to verify this inequality for $0 < s \leq s_0$. Since $(P^{\tau})'(T^-) < 0$, then $(F^{\tau}_8)'(\mathbf{\Psi}(0)) < 0$
for all $\varepsilon$ sufficiently small. Therefore,
$F^{\tau}_8(\mathbf{\Psi}(s)) < 0$ for $0 < s \leq s_0$ and $s_0$ sufficiently small.

%
%

\end{proof}

We can now obtain a uniform estimate of solutions.

\begin{prop}
Consider a solution $\widehat{\mathbf{w}}$ of problem \eqref{Eq stat tau}, \eqref{Eq stat tau boundaries} with condition \eqref{Eq stat tau mono}. Then
$$\widehat{\mathbf{w}}(x)\leq \mathbf{w}^- \;\; \text{ for } \;\; x\geq 0.$$
\label{lem w^-}
\end{prop}
\begin{proof}
Since $\widehat{\mathbf{w}}$ is decreasing, we only need to verify this result for $\hat{\mathbf{w}}(0)$. Let us set $\mathbf{w}_{\ell}=(w_{1,\ell},...,w_{7,\ell},T^-)$ where $w_{i,\ell}=\underset{s\rightarrow +\infty}{\text{lim}} \phi_i(s)$ for $1\leq i \leq 7$. We claim that:
\begin{equation}
\widehat{\mathbf{w}}(0)<\mathbf{w}_{\ell}.
\label{Ine w<wl}
\end{equation}
We will first show that $\widehat{w}_i$ is bounded by $w_{i,\ell}$ for $i=1,...,7$.
Suppose that this is not the case and $\hat{w}_i(0)\geq w_{i,\ell}$ for some $i$.
Since $\underset{s\rightarrow +\infty}{\text{lim}} \phi_i(s)=w_{i,\ell}$, and $\phi_i(s)$ is increasing for $s>0$, then
$\widehat w_i(0)>\phi_i(s)$ for $s>0$.

Let us now proceed in the same order as for the computation of the functions $\phi_i$, and show that the former inequality leads to a contradiction. Suppose that for $i=3$ we have the inequality $w_3(0)>\phi_3(s)$ for all $s > 0$. It follows that  $F_3(\widehat{\mathbf{w}}(0)) < 0$. From  the equation for the component $\widehat w_3(x)$ it follows that $\widehat w_3''(0)> 0$. This contradicts the monotonicity of the solution.
%
%
The same approach applied the others components proves inequality \eqref{Ine w<wl}.
Next, assume that $\widehat{w}_8(0)\geq T^-$. Since $\widehat{w}_i(0) < w_{i,\ell}$, it follows that $F_8^{\tau}(\widehat{\mathbf{w}}(0)) < 0$.
As before, inequality $\widehat w_8''(0)> 0$  leads to the contradiction.

In order to prove Proposition \ref{lem w^-}, we consider the
function $\Psi(s)$ from Lemma \ref{lem Psi}. From \eqref{Ine w<wl} we have the inequality $\underset{s\rightarrow +\infty}{\text{lim}}\Psi(s)=\mathbf{w}_{\ell}> \hat{\mathbf{w}}(0)$. Hence it exists $\sigma$ such that $\Psi(\sigma)\geq \widehat{\mathbf{w}}(0)$. Denote $\mathbf{w}^{\sigma}=\Psi(\sigma)$. From Lemma \ref{lem Psi} it follows that $\mathbf{F}^{\tau}(\mathbf{w}^{\sigma})<0$.

Consider the domain
\begin{equation}
D(\mathbf{w}^{\sigma})=\{\mathbf{w} | w_i^+\leq w_i \leq w_i^{\sigma},i\in\{1,...,7\}, w_8<T^- \}
\end{equation}
and its boundaries $\Gamma_i(\mathbf{w}^{\sigma})=\{\mathbf{w} | w_i=w_i^{\sigma}, w_j^+\leq w_j \leq w_j^{\sigma}, j\ne i\}$ for $i=1,...,7$.

Since $\widehat{\mathbf{w}}(0)<\mathbf{w}^{\sigma}$ and $\widehat{\mathbf{w}}(x)$ is decreasing, it follows that $\widehat{\mathbf{w}}\in D(w^{\sigma})$.
Thanks to the monotony of $\mathbf{F}^{\tau}$ given in \eqref{F tau monotone}, for every $\mathbf{w}\in \Gamma_i(\mathbf{w}^{\sigma})$ we have $F_i^{\tau}(\mathbf{w})\leq F_i^{\tau}(\mathbf{w}^{\sigma})<0$.

We now decrease the value of $\sigma$.
If the lemma is not true, then it exists $\sigma$ and a component $\widehat{w_i}(0)>w_i^-$ with $\widehat{w_i}\in \Gamma_i(\mathbf{w}^{\sigma})$.
 We fix $\sigma$, and the corresponding $\mathbf{w}^{\sigma}$,  the minimal value for which $\widehat{w_i}\in \Gamma_i(\mathbf{w}^{\sigma})$ holds.
 Then $\widehat{w_i}(0)=w_i^{\sigma}$ and $\widehat{w_j}(0)\leq w_j^{\sigma}$ for all $j\ne i$, $j<8$.
The properties of $D(\mathbf{w}^{\sigma})$ assure that $F_i^{\tau}(\widehat{\mathbf{w}}(0))<0$. Since $\widehat{\mathbf{w}}$ is a solution \eqref{Eq stat tau}, it leads to $\widehat{w_i}''(0)>0$ which is impossible since $\hat{w_i}'(0)=0$ and $\hat{w_i}'(x)<0$ for $x>0$.

\end{proof}

We will now obtain estimates of solutions in H\"older spaces.

\begin{thm}\label{thm bound weighed}
 Let condition \eqref{Condition P} be satisfied and the speed $c^{\tau}$ in problem \eqref{Eq wave tau equation}, \eqref{Eq wave tau condition} be positive for all $\tau$ in $[0,1]$.
Then there exists a constant $R$ such that for all $\tau$ and all solutions $\mathbf{w}$ of problem \eqref{Eq stat tau},\eqref{Eq stat tau boundaries} with condition \eqref{Eq stat tau mono} the following estimate holds:
\begin{equation}\label{est}
  \|\mathbf{w}\|_{E^1_{\mu}}\leq R .
\end{equation}
\end{thm}

\begin{proof}
From the uniform estimate of solution given by
Proposition \ref{lem w^-} it easily follows that solutions of problem \eqref{Eq stat tau}, \eqref{Eq stat tau boundaries} with condition \eqref{Eq stat tau mono}  are uniformly bounded in the H\"older space without weight. Hence to prove the theorem it is  sufficient to prove that $\sup_x|\mathbf{w}(x)\mu(x)|$ is uniformly bounded.

The solutions decay exponentially at infinity. Therefore, the weighted norm $\sup_x|\mathbf{w}(x)\mu(x)|$ is bounded for each solution. Suppose that the solutions are not uniformly bounded in the weighted norm. Then there exists a sequence $\mathbf{w}^k$ of solutions of \eqref{Eq stat tau} such that:
\begin{equation}
\sup_{x\geq 0}|\mathbf{w}^k(x)\mu(x)| \rightarrow + \infty \quad \text{ as } k\rightarrow +\infty.
\end{equation}
These solutions can correspond to different values of $\tau$.

Let $\varepsilon>0$ be small enough, so that exponential decay of the solutions gives the existence of a constant $M$, independent of $k$, such that the estimate
\begin{equation}
|\mathbf{w}^k(x)\mu(x)|\leq M.
\end{equation}
follows from the inequality $|\mathbf{w}^k(x)|\leq \varepsilon$.
We choose $\varepsilon$ small enough, such $|\mathbf{w}^k(0)| > \varepsilon$ (see inequality \eqref{Eq eta} below). Then there exists $x_k>0$ such that $|\mathbf{w}^k(x_k)|=\varepsilon$. Then we have:
\begin{equation}
|\mathbf{w}^k(x)\mu(x-x_k)|\leq M \quad \text{ for } x\geq x_k.
\label{Majoration w^kmu}
\end{equation}

If the sequence of $x_k$ is bounded, then $|\mathbf{w}^k(x)\mu(x)|$ is uniformly bounded. Indeed, for $0\leq x \leq x_k\leq \sup_k(x_k) < +\infty$ we have $\mathbf{w}^k<w^-$ from Lemma \ref{lem w^-}, and the estimate \eqref{Majoration w^kmu} holds for $x\geq \sup_k(x_k)$.

Suppose now that $\sup_k(x_k)=+\infty$. Then we consider a subsequence of $x_k$, still written $x_k$, such that $x_k\rightarrow +\infty$. Consider the sequence of functions $\mathbf{z}^k(x)=\mathbf{w}^k(x+x_k)$. We can extract a subsequence that converges to a function $\mathbf{z}^0$ in $\mathcal{C}^2_{loc}(\R)$. Then the function $\mathbf{z}^0$ is monotonically decreasing, it is defined on $\R$, and it satisfies the equation
\begin{equation}
D\mathbf{z}^{0''}+\mathbf{F}^{\tau_0}(\mathbf{z}^0)=0,
\end{equation}
for some value $\tau_0\in[0,1]$. The estimate \eqref{Majoration w^kmu} still holds, hence, $\mathbf{z}^0(+\infty)=\mathbf{w}^+$.
Moreover we have $\mathbf{z}^0(-\infty)=\mathbf{z}^-$, where $\mathbf{F}^{\tau_0}(\mathbf{z}^-)=0$ and $\mathbf{z}^-\ne \mathbf{w}^+$ since $|\mathbf{z}^0(0)|=\varepsilon$.
 Then $\mathbf{z}^0$ is solution of \eqref{Eq wave tau equation} for $c^{\tau_0}=0$. Since $\mathbf{F}^{\tau_0}$ possesses exactly three stationary points, $\mathbf{z}^-=\mathbf{w}^- \text{ or } \bar{\mathbf{w}}$.
 Both cases lead to a contradiction.

Indeed, if $\mathbf{z}^-$ is unstable ($\mathbf{z}^-=\bar{\mathbf{w}}$), then a solution of \eqref{Eq wave tau equation} exists if and only if $c^{\tau}<0$ \cite{TravelingWave}. If $\mathbf{z}^-$ is stable ($\mathbf{z}^-=\mathbf{w}^-$), then $\mathbf{z}^0$ is a solution of \eqref{Eq wave tau equation} and \eqref{Eq wave tau condition} with $c^{\tau_0}=0$, which contradicts the assumption of Theorem \ref{thm bound weighed}, that is, $c^{\tau}>0$ for all $\tau$. Hence the function $\mathbf{z}^0$ can not exist, and $\sup_k(x_k)<+\infty$, completing the proof
of the uniform estimate in the weighted space. Estimate \eqref{est} can now be easily proved by conventional methods.

\end{proof}


\subsection{Separation of monotone solutions}
Problem \eqref{Eq stat tau}, \eqref{Eq stat tau boundaries} can have monotone solutions, satisfying condition \eqref{Eq stat tau mono} and called pulses, and non monotone solutions. Since we are looking for monotone solutions, we need to assure that the Leray-Schauder method can be applied to this kind of solutions. The idea is to construct an open subset of $E^1_{\mu}$ containing all pulses but such that the non monotone solutions are not in its closure. We will prove that the monotone and non monotone solutions are separated in the function space.

\begin{thm}\label{thm separation}
 Suppose that all monotone solutions of \eqref{Eq stat tau}, \eqref{Eq stat tau boundaries} are uniformly bounded in $ E^1_{\mu}$. Then there exists $r>0$ such that for all monotone solution $\mathbf{w}^M$ and all non monotone solution $\mathbf{w}^N$ the estimate
$$\|\mathbf{w}^M-\mathbf{w}^N\|_{E^1_{\mu}}\geq r  $$
holds.
\end{thm}

\begin{remark}
The existence of the uniform bound in $E^1_{\mu}$ is guaranteed by Theorem \ref{thm bound weighed}.
\end{remark}

To prove Theorem \ref{thm separation}, we  need beforehand two lemmas. First, let us prove that a non negative solution of \eqref{Eq stat tau}, \eqref{Eq stat tau boundaries} is either positive or identically zero.


\begin{lem}
Let $\widehat{\mathbf{w}}(x)$ be a non negative solution of problem \eqref{Eq stat tau}, \eqref{Eq stat tau boundaries} for some $\tau\in[0,1]$. Then one of the two following conclusions holds:
\begin{itemize}
\item $ \widehat{\mathbf{w}}(x) \equiv 0$ ,
\item $\widehat{\mathbf{w}}(x) >0$, $x \geq 0$ .
\end{itemize}
\label{lem positivity}
\end{lem}
\begin{proof}
Let $\widehat{\mathbf{w}}$ be a non negative solution. Suppose that there exists $\sigma \geq 0$ and a component of the solution such that $\widehat{w}_i(\sigma)=0$. Since the solution $\widehat{\mathbf{w}}$ is non negative we have $\widehat{w}_i'(\sigma)=0$ and $\widehat{w}_i''(\sigma)\geq 0$. Then from the corresponding equation of system \eqref{Eq stat tau} we conclude that $F_i^{\tau}(\widehat{\mathbf{w}}(\sigma))\leq 0$.  Using the expression of $\mathbf{F}^{\tau}$ given by \eqref{Def F}, \eqref{Equation P tau} and \eqref{Def Hom 3} we see that $F_i^{\tau}(\mathbf{w})=G_i^{\tau}(\mathbf{w})-\lambda_i(\mathbf{w}) w_i$, for some positive function $\lambda_i$ which does not depends on $w_i$ and where $G_i^{\tau}(\mathbf{w})\geq 0$ for $\mathbf{w}\geq \mathbf{0}$. Since $\widehat{w}_i(\sigma)=0$, it follows $F_i^{\tau}(\widehat{\mathbf{w}}(\sigma))=0$.

Suppose, for example, that $i=1$. Then from the equality $F_i^{\tau}(\widehat{\mathbf{w}}(\sigma))=0$ it follows that
$\widehat{w}_3(\sigma)=0$ or $\widehat{w}_6(\sigma)=0$. We repeat the same arguments, respectively, for $i=3$ or $i=6$
and conclude that some other components vanish at $x=\sigma$. Finally, we get $\widehat{\mathbf{w}}(\sigma)=0$. From the uniqueness of solution
of the Cauchy problem we conclude that $\widehat{\mathbf{w}}(x)\equiv \mathbf{0}$. The proof remains similar for all other values of $i$
for which the corresponding component of the solution vanish.

%
%

\end{proof}

The next result concerns the Jacobian matrix  $(\mathbf{F}^{\tau})'$.

\begin{lem}
There exists a constant vector $\mathbf{q}>\mathbf{0}$ such that $(\mathbf{F}^{\tau})'(\mathbf{0}).\mathbf{q}<\mathbf{0}$
$\forall \tau\in[0,1]$.
\label{lem q<0}
\end{lem}

\begin{proof}
The expression of the Jacobian matrix $(\mathbf{F}^{\tau})$ is given in the Appendix \ref{App Jac Mat}. We notice that for $(\mathbf{F}^{\tau})'(\mathbf{0})$ we have $\theta_{1,3}=\theta_{1,6}=\theta_{2,4}=\theta_{2,5}=0$. The other elements of this matrix, $\theta_{i,j}$ and $d_i$ are strictly positive, while $H^{\tau}(0)=\beta^{\tau}P'(0)-\alpha^{\tau}h_8 <0$.
Direct calculations give:
\begin{equation}
(\mathbf{F}^{\tau})'(0).\mathbf{q}=
\begin{pmatrix}
-d_1q_1 \\
-d_2q_2 \\
-d_3q_3+\theta_{3,8}q_8 \\
-d_4q_4+\theta_{4,8}q_8 \\
-d_5q_5+\theta_{5,7}q_7 \\
-d_6q_6+\theta_{6,2}q_2+\theta_{6,5}q_5 \\
-d_7q_7+\theta_{7,8}q_8 \\
H^{\tau}(0)q_8+\theta_{8,1}q_1+\theta_{8,6}q_6
\end{pmatrix}.
\end{equation}
Inequality $(\mathbf{F}^{\tau})'(\mathbf{0}).\mathbf{q}<\mathbf{0}$ is satisfied if the following conditions hold:

 $$ q_1>0, \; q_2>0, \; q_3>\frac{\theta_{3,8}}{d_3}\;q_8, \; q_4>\frac{\theta_{4,8}}{d_4}\;q_8 \; , $$

 $$ q_5>\frac{\theta_{5,7}}{d_5}\;q_7 , \;
 q_6>\frac{\theta_{6,2}q_2+\theta_{6,5}q_5}{d_6}, \; q_7>\frac{\theta_{7,8}}{d_7} \; q_8, \; q_8>\frac{\theta_{8,1}q_1+\theta_{8,6}q_6}{H^{\tau}(0)} \; . $$
The first four inequalities are independent and the last four inequalities lead to the condition:
\begin{equation}
q_5>\frac{\theta_{5,7}\theta_{7,8}}{d_5d_7}\left(\frac{\theta_{8,1}}{H^{\tau}(0)}q_1+
\frac{\theta_{8,6}}{H^{\tau}(0)}\left(\frac{\theta_{6,2}}{d_6} \; q_2+\frac{\theta_{6,5}}{d_6} \; q_5\right)\right).
\end{equation}
It can be satisfied if  the following estimate holds:
\begin{equation}
\frac{\theta_{5,7}\theta_{7,8}}{d_5d_7}\frac{\theta_{8,6}}{H^{\tau}(0)}\frac{\theta_{6,5}}{d_6}<1.
\label{Condition q>0}
\end{equation}
It is equivalent to the condition:
\begin{equation}
\label{z1}
\beta^{\tau}P'(0)h_5h_6h_7<-\alpha^{\tau}(k_5\rho_5k_6\rho_6k_7\rho_7k_8\rho_8-h_5h_6h_7h_8).
\end{equation}
The right-hand side of this inequality equals $-\alpha^{\tau}\frac{d}{h_3h_4}$ (see \eqref{Coef d}).
Since $d<0$ due to conditions \eqref{Condition Q}, then this expression is positive.
To verify \eqref{z1}, it remains to note that $P'(0)<0$.
\end{proof}

 We now return to the proof of Theorem \ref{thm separation}. Let us consider a sequence of monotone solutions $(\mathbf{w}^{M,k})_{k\geq0}$ of problem \eqref{Eq stat tau}, \eqref{Eq stat tau boundaries} and  a sequence of non monotone solutions $(\mathbf{w}^{N,k})_{k\geq0}$.
 First, we will prove that the set of monotone solutions is closed. After that we will show that a squence of non monotone solutions  cannot converge to a monotone solution.

Let us start with the first step. Since $(\mathbf{w}^{M,k})_{k\geq0}$ is bounded in $E_{\mu}^1$, then it is compact in $E^1$.
Therefore, we can extract a subsequence of $(\mathbf{w}^{M,k})_{k\geq0}$, still written $(\mathbf{w}^{M,k})_{k\geq0}$, that converges in $E^1$. We call $\widehat{\mathbf{w}}$ its limit.
Thus we have

$$D \widehat{\mathbf{w}}''+\mathbf{F}^{\hat{\tau}}(\widehat{\mathbf{w}})=0$$
for some $\hat{\tau} \in [0,1]$. Moreover, $\widehat{\mathbf{w}}\geq 0$, $\widehat{\mathbf{w}}'\leq 0$.
 We will show that the solution $\widehat{\mathbf{w}}$ is a pulse, that is, it satisfies problem \eqref{Eq stat tau}, \eqref{Eq stat tau boundaries} with condition \eqref{Eq stat tau mono}. First, let us prove that $\widehat{\mathbf{w}}$ remains positive.

\begin{lem} The inequality $\widehat{\mathbf{w}}(x)>0$ holds for all $x \geq 0$.
\label{lem w>0}
\end{lem}

\begin{proof}
To prove this lemma, it is sufficient to show that $\widehat{\mathbf{w}}(0)>0$ and to apply the Lemma \ref{lem positivity}. We remark that if at least one component of the vector $\widehat{\mathbf{w}}(0)$ equals zero, then the whole vector vanish, $\widehat{\mathbf{w}}(0)=\mathbf{0}$.
 If this is the case, then
  $\mathbf{w}^{M,k}(0)\rightarrow \mathbf{0}$. Lemma \ref{lem q<0} assures the existence of a positive vector $\mathbf{q}$ such that for all $\tau\in[0,1]$, $(\mathbf{F}^{\tau})'(0)\mathbf{q}<\mathbf{0}$. Let $\varepsilon>0$. Then the following expansion holds:
  $$ \mathbf{F}^{\hat{\tau}}(\varepsilon \mathbf{q})=\mathbf{F}^{\hat{\tau}}(0)+\varepsilon(\mathbf{F}^{\hat{\tau}}))'(0).\mathbf{q}+o(\varepsilon \mathbf{q})  $$
  with $\mathbf{F}^{\hat{\tau}}(0)=\mathbf{0}$ and $(\mathbf{F}^{\hat{\tau}})'(0).\mathbf{q}<\mathbf{0}$. Hence $\mathbf{F}^{\hat{\tau}}(\varepsilon \mathbf{q})<\mathbf{0}$ for $\varepsilon$ small enough.

  Due to the continuity of the function $\mathbf{F}^{\tau}$ with respect to $\tau$, the inequality $\mathbf{F}^{\tau}(\varepsilon \mathbf{q})<\mathbf{0}$ holds for all $\tau$ close to $\hat{\tau}$.
  Furthermore the monotony property \eqref{F tau monotone} assures that for any $\mathbf{w} \in B_{\varepsilon}=[\mathbf{0},\varepsilon \mathbf{q}]$, $\mathbf{w}\ne 0$, at least one component of the vector $\mathbf{F}^{\tau}(\mathbf{w})$ is negative.

Thus, for some component $i$ and for some element $k$ of the  sequence, we have the inequality $F_i^{\tau_k}(\mathbf{w}^{M,k})<0$.
Then from \eqref{Eq stat tau}, we get $w_i^{M,k''}(0)>0$ implying that the function $w_i^{M,k}(x)$ is not decreasing (since $\mathbf{w}^{M,k'}(0)=\mathbf{0}$), which contradicts \eqref{Eq stat tau mono}.
 Thus, $\widehat{\mathbf{w}}(0)\ne \mathbf{0}$. Lemma \ref{lem positivity} guarantees the positiveness of $\widehat{\mathbf{w}}(x)$ for all $x \geq 0$.
\end{proof}

Let us note that this lemma implies that pulse solutions are separated from the trivial solution $\mathbf{w}\equiv 0$:
\begin{equation}\label{Eq eta}
\exists \boldsymbol{\eta} > 0 \; \text{such that for any solution} \; \mathbf{w} \; \text{ of \eqref{Eq stat} }, \mathbf{w}(0)>\boldsymbol{\eta}.
\end{equation}


\medskip

\begin{lem} The inequality $\widehat{\mathbf{w}}'(x)<\mathbf{0}$ holds for all $x > 0$.
\label{lem w'<0}
\end{lem}

\begin{proof}
Suppose that there exits a component $\widehat{\mathbf{w}}$ of the solution and a value $x_0$ such that $\widehat{w}'_i(x_0)=0$. Denote $v=-\widehat{w}_i'$. Then $v\geq 0$
since $\widehat{\mathbf{w}}$ is a monotone solution of problem \eqref{Eq stat tau}, \eqref{Eq stat tau boundaries}. Differentiation of the $i$th of system \eqref{Eq stat tau} gives the following equality:
\begin{equation}
{D_i}v''(x)+v(x)\frac{\partial F_i^{\hat{\tau}}}{\partial v_i}(\mathbf{\widehat{w}}(x))=\sum_{j\ne i}\frac{\partial F_j^{\hat{\tau}}}{\partial v_i}(\mathbf{\widehat{w}}(x))\hat{w}_j'(x).
\label{Eq for v}
\end{equation}
We introduce the operator

$$Lv={D_i}v''+v\frac{\partial F_i^{\hat{\tau}}}{\partial v_i}(\mathbf{\widehat{w}}) . $$
Since $\mathbf{F}^{\hat{\tau}}$ satisfies the monotonicity condition and, according to Lemma \ref{lem w^-} $\mathbf{\widehat{w}}$ is in $\mathcal{C}$ (defined in \eqref{Def C}), then $\frac{\partial F_j^{\hat{\tau}}}{\partial v_i}(\mathbf{\widehat{w}}(x)) < 0$  for all $j \neq i$.
 Moreover, the solution $\widehat{\mathbf{w}}$ is decreasing. Consequently, the right-hand side of \eqref{Eq for v} is positive. Thus, $Lv\geq 0$.
  The operator $L$ is elliptic and $v$ reaches its minimum at $x_0$. Hence, the maximum principle states that $v\equiv 0$.
  Since $\widehat{\mathbf{w}}$ vanishes at infinity, it follows that $\widehat{w}_i\equiv 0$, which is impossible.
\end{proof}

Lemmas \ref{lem w>0} and \ref{lem w'<0} prove that $\widehat{\mathbf{w}}$ is a pulse, that is, a solution of problem \eqref{Eq stat tau}, \eqref{Eq stat tau boundaries} with condition \eqref{Eq stat tau mono}. Hence the set of pulse solutions is closed in $E^1_{\mu}$.
We will now conclude the proof of Theorem \ref{thm separation} and prove that a sequence of non monotone solutions cannot converge to a monotone solution.

\begin{proof}
Let us assume that the separation between monotone and non monotone solutions does not hold. Then we can find a sequence of monotone solutions $(\mathbf{w}^{M,k})$ and a sequence of non monotone solutions $(\mathbf{w}^{N,k})$ such that $\|\mathbf{w}^{M,k}-\mathbf{w}^{N,k}\|_{E^1_{\mu}}\rightarrow 0$. As it is shown in Lemmas \ref{lem w>0} and \ref{lem w'<0}, we can extract a subsequence from $(\mathbf{w}^{M,k})$, still denoted by $(\mathbf{w}^{M,k})$, that converges to some pulse solution $\widehat{\mathbf{w}}$. Then we have the convergence  $\mathbf{w}^{N,k}\rightarrow \widehat{\mathbf{w}}$, possibly for a subsequence. Next,  we can extract a subsequence, for which we keep the same notation, for which some given component $i$ is non monotone. The solution $\mathbf{w}^{N,k}$ belongs to $\mathcal{C}^1$. Hence there exists a sequence $x_k>0$ such that $\mathbf{w}^{N,k'}(x_k)=\mathbf{0}$. Then there exists a subsequence, still written as $x_k$, such that $x_k\rightarrow x_0$. There are three possible values for the limit $x_0$:
\begin{itemize}
\item $x_0\in ]0,+\infty[$,
\item $x_0 = +\infty$,
\item $x_0=0$.
\end{itemize}
We will show that none of them is possible.
If $x_0\in ]0,+\infty[$, then $\widehat{w}_i'(x_0)=0$. We obtain a contradiction with Lemma 3.8.

 Next, we consider the case $x_0 = +\infty$. We claim that for $k$ and $y$ large enough the non monotone solution is decreasing: $\mathbf{w}^{N,k'}(x) <\mathbf{0}$ for $x\geq y$. Consequently, we will have $x_0<+\infty$ reducing this case to the previous one.
According to Lemma \ref{lem q<0}, there exits a vector $\mathbf{q}>\mathbf{0}$ such that for all $\tau \in [0,1]$ we have $(\mathbf{F}^{\tau})'(0).\mathbf{q}<\mathbf{0}$.
  Since the function $(\mathbf{F}^{\tau})'$ is continuous with respect to $\tau$, then there exist $\boldsymbol{\delta}>0$ and $\varepsilon > 0$ such that for $\tau$ and $\mathbf{w}\in \R^8$ satisfying $|\tau-\hat{\tau}|<\varepsilon$ and $\|\mathbf{w}\|_{\R^8}<\boldsymbol{\delta}$, we have the inequality $(\mathbf{F}^{\tau})'(\mathbf{w}).\mathbf{q}<\mathbf{0}$.

 The solution $\widehat{\mathbf{w}}$ is decreasing (Lemma \ref{lem w'<0}) and converges to $0$ at infinity. Hence there exists $y>0$ such that for all $x\geq y$ we have the estimate $\|\widehat{\mathbf{w}}(x)\|_{\R^8}<\boldsymbol{\delta}$. The sequence $(\mathbf{w}^{N,k})$ converges to $\widehat{\mathbf{w}}$ in $E_{\mu}^1$. Then for $k_1$ large enough and $k\geq k_1$ we have $|\tau_k-\hat{\tau}|<\varepsilon$ and $|\mathbf{w}^{N,k}(x)\|_{\R^8}<\boldsymbol{\delta}$ for $x \geq y$. Moreover, $(\mathbf{w}^{N,k})'(y)<\mathbf{0}$.

 Let us show that $(\mathbf{w}^{N,k})'(x)$ remains negative for $x \geq y$. We differentiate \eqref{Eq stat tau}  and denote $\mathbf{z}^k(x)=-(\mathbf{w}^{N,k})'(x)$. Then
 \begin{equation}
 \label{zzz}
 D\mathbf{z}^{k''}+\mathbf{F}^{\tau_k'}(\mathbf{w}^{N,k}).\mathbf{z}^k=0.
 \end{equation}
 Suppose that there exist $k\geq k_1$ and $x>y$ such that ${z_i}^k(x)< {0}$ for some component of the function $\mathbf{z}^k(x)$. Since $\mathbf{z}^k(y)>\mathbf{0}$ and $\mathbf{z}^k(\infty)=\mathbf{0}$, then there exists a constant $a$ such that $\boldsymbol{\zeta}^k(x)=\mathbf{z}^k(x)+a\mathbf{q}\geq0$, and a value $s>y$ such that ${\zeta_i}^k(s)=0$.
 Then the function ${\zeta_i}^k(x)$ satisfies the equation:
 \begin{equation}
 {D_i} {\zeta_i}^{k''}+ {F_i}^{\tau_k'}(\mathbf{w}^{N,k}){\zeta_i}^k+a_k=0,
 \end{equation}
where $a_k=-a \sum_{j \neq i} {F_j}^{\tau_k'}(\mathbf{w}^{N,k}){q_j}>\mathbf{0}$ since $k>k_1$ and $x>y$. Thus $({\zeta_i}^{k})''(s)<0$, which contradict the fact that $s$ is a minimum.

Finally, we consider the case $x_0=0$. Since $\widehat{\mathbf{w}}$ is a solution of problem \eqref{Eq stat tau}, \eqref{Eq stat tau boundaries} with conditio  \eqref{Eq stat tau mono}, it follows that $\mathbf{F}^{\widehat{\tau}}(\widehat{\mathbf{w}}(0))\geq \mathbf{0}$. First, let us prove that
\begin{equation}
\mathbf{F}^{\hat{\tau}}(\widehat{\mathbf{w}}(0))> \mathbf{0}.
\end{equation}
Suppose that this is not true. Then there exists a component of the solution such that $F_i(\widehat{\mathbf{w}}(0))=0$. Set $z(x)=-\widehat{w}_i'(x)$, then we differentiate \eqref{Eq wave tau equation} in order to obtain:
\begin{equation}
{D_i}z''(x)+\frac{\partial F_i^{\hat{\tau}}}{\partial v_i}(\widehat{\mathbf{w}}(x))z(x)+a(x)=0,
\end{equation}
where $a(x)=-\sum_{j\ne i}\frac{\partial F_i^{\hat{\tau}}}{\partial v_j}(\widehat{\mathbf{w}}(x))\widehat{w}_j'(x) \geq 0$. Since $z(0)=0$ and $z'(0)=0$, it leads to a contradiction with the Hopf lemma.

The functions $\mathbf{w}^{N,k}(x)$ converge to $\widehat{\mathbf{w}}(x)$. Consequently for $k$ large enough we have $\mathbf{F}^{\tau_k}(\mathbf{w}^{N,k}(0))>\mathbf{0}$. Hence, there exist $\varepsilon>0$ small enough and a constant $k_1$ such that for $x\in]0,\varepsilon[$ and $k\geq k_1$ we have $\mathbf{F}^{\tau_k}(\mathbf{w}^{N,k}(x))>\mathbf{0}$. Hence $(\mathbf{w}^{N,k})'(x)<\mathbf{0}$ in the interval $]0,\varepsilon[$, and the convergence $x_k\rightarrow 0$ does not hold.


\end{proof}

In this section we obtained a priori estimates of monotone solutions  and proved that they are separated from non monotone solutions. In the next section we will prove  Theorem \ref{Thm principal}.

\section{Proof of the main theorem}
\label{Sec Proof}

In this section the main Theoreme \ref{Thm principal} will be proved. The first part of the proof will be devoted to the existence of pulse solutions if the wave speed is positive. The function $g$ defined at the first step of the homotopy \eqref{Def Hom 1} will be established in order to preserve the positiveness of the wave speed along the homotopy. Then we will prove that the pulse problem \eqref{Eq stat tau}, \eqref{Eq stat tau boundaries} with condition \eqref{Eq stat tau mono} possesses a solution for $\tau=1$. Finally, we will find the value of the topological degree of the operator $A^1$, and we will show that it does not change along the homotopy. We will use here Theorems \ref{thm bound weighed} and \ref{thm separation} presented in Section \ref{Sec Estimation}. The second part of the proof will be devoted to the non existence of solution of problem \eqref{Eq stat} if the wave speed $c$ in  problem \eqref{Eq wave} is non positive.

\subsection{Positiveness of the wave speed}
Let us start with the preservation of the positiveness of the wave speed.
We assume that the wave speed $c$, also denoted as $c^{0}=c$, in problem \eqref{Eq wave} is positive:
\begin{equation}
c>0.
\end{equation}
We will construct a function $g$ that satisfies \eqref{condition g} and which will provide a positive wave speed along the homotopy.

 We begin with $\tau \in ]0,\tau_1]$. The definition of the homotopy \eqref{Def Hom 1} leads to the inequalilty $F_8^{\tau}>F_8^0$ independently of the choice of $g$. Consequently, using the result of \cite{Spectrum} (pages 111-134), we have $c^{\tau}\geq c^0$. Hence
\begin{equation}
\forall \tau \in [0,\tau_1], \ c^{\tau}>0.
\end{equation}

Let us consider $\tau \in ]\tau_1,1]$. The homotopy on this interval is defined by equality \eqref{Def Hom 2}.  For $\mathbf{w}\in \mathcal{C}$ and for $\tau\in [\tau_1,1]$ we have
\begin{equation}
F_8^{\tau}(\mathbf{w})\geq \tau_1 g(T)-h_8T \equiv G(T).
\end{equation}
The condition \eqref{condition g} on $g(T)$ leads to $G(0)=0$ and $G'(0)=-h_8<0$. Then we can find a function $g(T)$ such that $G(T)$ satisfies the following conditions:
\begin{equation}
 \begin{cases}
 G(T) \text{ possesses exactly three zeros in } [0,T^-]: 0,\ T_1, \ T_2, \; \\ 0 < T_1 < T_2 < T^- , \\
 G(T)<0 \text{ for } 0<T<T_1,\ G(T)>0 \text{ for } T_1<T<T_2,\ \\ G(T)<0 \text{ for } T_2<T<T^-.
 \end{cases}
 \label{condition g 2}
 \end{equation}

We can now establish the condition providing the positiveness of the wave speed along the homotopy:

\begin{prop}
Suppose that function $g(T)$ satisfies conditions \eqref{condition g}, \eqref{condition g 2}. Furthermore, assume that
\begin{equation}
\int_0^{T_2}G(s)ds>0 .
\label{condition g 3}
\end{equation}
Then $c^{\tau}>0$ for $\tau \in [\tau_1,1]$.
\end{prop}

\begin{proof}
Let us consider the scalar parabolic equation
\begin{equation}\label{Eq para sca}
\frac{\partial \theta}{\partial t}=\frac{\partial^2 \theta}{\partial x^2}+G(\theta)
\end{equation}
on the whole axis with the initial condition $\theta(x,0)=\theta^0(x)$, where $\theta^0(x)$ is a monotonically decreasing function, $\theta^0(+\infty)=T^+=0$ and $\theta^0(-\infty)=T_1$. Since the function $G(T)$ satisfies the condition \eqref{condition g 2}, the solution of \eqref{Eq para sca} converges to a traveling wave $\theta_1(x-c_1t)$. The wave speed $c_1$ has the sign of the integral \eqref{condition g 3}. Hence, by hypothesis, $c_1>0$.

Next, we claim that
\begin{equation}
c^{\tau}>c_1, \; \forall \tau \in [\tau_1,1] .
\end{equation}
In order to prove this inequality, we consider the parabolic problem:
\begin{equation}\label{Eq para pb}
\frac{\partial \mathbf{v}^{\tau}}{\partial t}=\frac{\partial^2 \mathbf{v}^{\tau}}{\partial x^2}+\mathbf{F}^{\tau}(\mathbf{v}^{\tau}), \quad \mathbf{v}^{\tau}(x,0)=\mathbf{v}^{0}(x).
\end{equation}
Let us assume that $\mathbf{v}^{0}(-\infty)=\mathbf{w}^-$ and $\mathbf{v}^{0}(+\infty)=w^+$, then the solution of \eqref{Eq para pb} converges to the wave solution $\mathbf{u}^{\tau}(x-c^{\tau}t)$ with the wave speed $c^{\tau}$.

Similarly to \eqref{Eq para pb} we consider the problem
\begin{equation}\label{Eq para pb sca}
\frac{\partial \mathbf{z}^{\tau}}{\partial t}=\frac{\partial^2 \mathbf{z}^{\tau}}{\partial x^2}+\mathbf{H}^{\tau}(\mathbf{z}^{\tau}), \quad \mathbf{z}^{\tau}(x,0)=\mathbf{z}^{0}(x),
\end{equation}
where $H_i(\mathbf{v})=F_i^{\tau}(\mathbf{v})$ for $i=1,...,7$ and $H_8(\mathbf{v})=G(v_8)$. Furthermore, we assume that $z_8^0(x)=\theta^0(x)$ and $z_i^0(x)\equiv 0$ for $i=1,...,7$.

Let us assume that $v_8^0(x)\geq z_8^0(x)$ on the whole axis. Taking into account that $\mathbf{F}^{\tau}(\mathbf{v})\geq\mathbf{H}^{\tau}(\mathbf{v})$, and $\mathbf{v}^0\geq \mathbf{z}^0$, we conclude that $\mathbf{v}^{\tau}(x,t)\geq \mathbf{z}^{\tau}(x,t)$ for all $x\in \R$ and $t>0$. Since $\mathbf{v}^{\tau}$ converges to the wave with the speed $c^{\tau}$ and $\mathbf{z}^{\tau}$ to the wave with the speed $c_1$, then $c^{\tau}\geq c_1>0$.
\end{proof}

\medskip

We have shown that it is possible to find a function $g(T)$ satisfying \eqref{condition g} such that the wave speed $c^{\tau}$ remains positive along the homotopy. Let us fix such a function $g(T)$ and proceed to the resolution of the pulse problem for $\tau=1$.


\subsection{Scalar equation: the case $\tau=1$}

Consider problem \eqref{Eq stat tau}, \eqref{Eq stat tau boundaries} with condition \eqref{Eq stat tau mono} for $\tau=1$.
We have the equation for the last component $T$ independent of the other equations:
\begin{equation}\label{Eq stat T 1}
\begin{cases}
{D_8}T''+P(T)+\tau_1g(T)=0, \\
T'(0)=T(+\infty)=0, \\
T'(0)<0, \quad \text{ for } x>0.
\end{cases}
\end{equation}

\medskip

\begin{lem}\label{lem exit T}
 Let $P(T)$ satisfy \eqref{Condition P} and $g(T)$ satisfy \eqref{condition g}, \eqref{condition g 2}, \eqref{condition g 3}. Then problem \eqref{Eq stat T 1} possesses a unique solution.
\end{lem}

\begin{proof}
The second-order equation in \eqref{Eq stat T 1} can be written as the system of first-order equations:
\begin{equation}\label{SystT 1}
\begin{cases}
w'=p\\
p'=-f(w),
\end{cases}
\end{equation}
where $f(w)=({P(w)+\tau_1g(w)})/{{D_8}}$. 
From \eqref{SystT 1} it follows that:
\begin{equation}
pp'=-w'f(w).
\end{equation}
Integrating this equation, and taking into account that $p(0)=0$, we get:
\begin{equation}\label{eq p^2}
\frac{1}{2}p^2(w)=\int_{w}^{w(0)}f(u)du.
\end{equation}

The function $f$ only changes sign once on $]0,T^-[$ since $P$ and $g$ satisfy respectively \eqref{Condition P} and \eqref{condition g}. Moreover, \eqref{condition g 3} assures that the integral of $f$ on $[0,T^-]$ is strictly positive, and Lemma \eqref{lem w^-} guarantees that $w(0)<T^-$, hence it exists a unique value $w_0=w(0)$ satisfying:
\begin{equation}\label{Condition w_0}
\int_0^{w_0}f(u)du=0 , \;\;\; 0<w_0<T^- .
\end{equation}
Hence, the right-hand side of equation \eqref{eq p^2} is non-negative. Its solution provides a solution of system \eqref{SystT 1}.

\end{proof}

Thus scalar equation possesses a pulse solution.
We now consider system \eqref{Eq stat tau} for $\tau=1$. We claim the following result.

\begin{prop}\label{prop ex pulse}
Let $P(T)$ satisfy \eqref{Condition P} and $g(T)$ satisfy \eqref{condition g}, \eqref{condition g 2} and \eqref{condition g 3}. Then problem \eqref{Eq stat tau}, \eqref{Eq stat tau boundaries} with condition \eqref{Eq stat tau mono} possesses a unique solution for $\tau=1$.
\end{prop}

\begin{proof}
The last equation of \eqref{Eq stat tau} for $\tau=1$, that is the equation for $T$, is independent from the other equations. Lemma \ref{lem exit T} assures the existence and uniqueness of a monotone solution $T(x)$ satisfying \eqref{Eq stat T 1}.
Let us fix this function $T(x)$.

 In the rest of the proof, we will first show the existence and uniqueness of a positive solution of problem \eqref{Eq stat tau}, \eqref{Eq stat tau boundaries}, and then we will prove that this solution is monotone. In both steps we will proceed in the same order of components of solution as for the computation of the functions $\phi_i$.

First, let us focus on the existence of a positive solution of problem \eqref{Eq stat tau}, \eqref{Eq stat tau boundaries}. We start with the third component, and we have the following problem:
\begin{eqnarray}\label{Eq v_3 tau 1}
{D_3}w_3''-(k_3T+h_3)w_3=-k_3\rho_3T, \\
w_3'(0)=0. \nonumber
\end{eqnarray}
We introduce the operator
$$ Lu={D_3}u''-(k_3T+h_3)u , $$
and we are looking for a solution $v_3\in E^1_{\mu}$ of the equation $Lv_3=-k_3\rho_3T$. Then we have $Lv_3<0$ and $a_3(x)=k_3T(x)+h_3>h_3>0$. Hence, the operator $L$ is invertible, and there exists $\varepsilon>0$ such that for all $\lambda\in \text{ Spect }\{L\}$, we have the estimate $\Re \lambda <-\varepsilon$. Consequently, there exists a unique solution $w_3$ of problem \eqref{Eq v_3 tau 1}, and the solution satisfies the inequality $w_3(x)>0$ for all $x \geq 0$.

Proceeding with the same method for the other components (and in the same order as for the computation of $\phi_i$) we prove the existence and uniqueness of a solution $\mathbf{w}(x)$ to \eqref{Eq stat tau} with the boundary condition $\mathbf{w}'(0)=0$.

Let us now prove that this solution is monotonically decreasing. Lemma \ref{lem exit T} already states that $w_8'<0$. Once again we proceed in the same order for the components of solution, starting with the third component. Let $v(x)=-w_3(x)$. Differentiation of equation \eqref{Eq v_3 tau 1} leads to the equation:
\begin{equation}
{D_3}v''-a_3(x)v=-k_3\rho_3T'-k_3w_3T'.
\end{equation}
Setting $f_3=-k_3\rho_3T'-k_3w_3T'$ we note that $v$ satisfies the equation $Lv=f_3>0$. Hence, as for the existence of $w_3$, we get $v>0$.
Consequently it follows that ${w_3}'(x)<0$ for $x>0$. We repeat the same argument for the other components of solution.

Let us investigate the behavior of $\mathbf{w}(x)$ at infinity. Since $\mathbf{w}(x)$ is positive and decreasing, then there is a limit at $+\infty$. Consequently $\mathbf{w}''$ has also a limit, and this limit is $\mathbf{0}$. Then we have $\mathbf{F}^1(\mathbf{w(+\infty)})=\mathbf{0}$. Moreover, Lemma \ref{lem exit T} assures that $w_8(+\infty)=0$. 
Hence, $\mathbf{w(+\infty)}=\mathbf{0}$.
\end{proof}

We have proved that the problem \eqref{Eq stat tau}, \eqref{Eq stat tau boundaries} with condition \eqref{Eq stat tau mono} possesses exactly one solution for $\tau=1$. Now let us show that the topological degree is different from $0$ and that it is preserved along the homotopy.


\medskip

\subsection{Leray-Schauder method and the existence of pulses}

We us the construction of the topological degree for elliptic operators in unbounded domains in weighted H\"older spaces \cite{V1}.
In order to calculate its value,  we need to assure that the operator $A^{\tau}$ defined by \eqref{Def A tau} and linearized about this solution does not have a zero eigenvalue.
We consider the eigenvalue problem for system \eqref{Eq stat tau} linearized about a pulse solution $\mathbf{w}$:

\begin{equation}\label{Eq linearized}
D\mathbf{v}''+(\mathbf{F}^1(\mathbf{w}))'\mathbf{v}=\lambda\mathbf{v},
\end{equation}
on the half-axis $x>0$ with the boundary conditions:
\begin{equation}\label{cond linearized}
\mathbf{v}'(0)=0, \ \mathbf{v}(+\infty)=0.
\end{equation}
Here $v\in E^1_{\mu}$.
We claim the following result.

\begin{prop}\label{prop line eig}
All eigenvalues of the linearized problem \eqref{Eq linearized}, \eqref{cond linearized} are different from $0$.
\end{prop}

\begin{proof}
Let us assume that the assertion of the proposition does not hold. Then there exists a nonzero function $\mathbf{v}$ such that
 $$D\mathbf{v}''+(\mathbf{F}^1)'(\mathbf{w})\mathbf{v}=\mathbf{0} , $$
  and $\mathbf{v}'(0)=\mathbf{0}$,  $\mathbf{v}(+\infty)=\mathbf{0}$.
%

The last component of ${v_8}$ of the solution satisfies the equation
\begin{equation}\label{Eq v}
{D_8}v_8''+(F_8^1)'(T)v_8=0, \ v_8'(0)=0, \ v_8(\infty)=0.
\end{equation}
independent of other equations. We will show that it leads to a contradiction.
Let us note that $v_8(0) \neq 0$. Indeed, otherwise $v_8(x) \equiv 0$. In this case, it can be easily proved that all other components
of the solution are also identical zeros. Without loss of generality we can assume that $v_8(0) > 0$.

 Next, we differentiate equation \eqref{Eq stat T 1} and set $u=-T'$, where $T$ is the solution of problem \eqref{Eq stat T 1} given by Lemma \ref{lem exit T}.  It satisfies the problem
\begin{equation}\label{Eq u}
{D_8}u''+(F_8^1)'(T)u=0, \ u(0)=0, \ u(\infty)=0 .
\end{equation}
Moreover, $u(x) > 0$ for all $x > 0$.

 Let us recall that $(F_8^1)'(T)=P'(T)+\tau_1g'(T)$. Since $g(T)$ satisfies \eqref{condition g} and $P(T)$ satisfies \eqref{Condition P}, it follows that $(F_8^1)'(0)<0$. Furthermore, the function $T(x)$ converges to $0$ at infinity, so there exists $x_*>0$ such that $(F_8^1)'(T(x))<0$ for $x\geq x_*$. We need the following lemma to continue the proof of the proposition.


\begin{lem}
Suppose that solution $z(x) \not\equiv 0$ of the equation
\begin{equation}\label{lem intermediaire}
{D_8}z''+(F_8^1)'(T(x))z=0, \ z(\infty)=0.
\end{equation}
is such that $z(x_0) \geq 0$ for some $x_0 > 0$.
If $(F_8^1)'(T(x))<0$ for $x \geq x_0$, then $z(x)>0$ for $x \geq x_0$.
\end{lem}
\begin{proof}
Suppose that the assertion of the lemma does not hold. Then there exists $x_1>x_0$ such that $z(x_1)\leq 0$. Since $z(\infty)=0$, then $z(x)$ admits a minimum $x_2\geq x_1$. If $z(x_2)=0$, then $z(x_2)=z'(x_2)=0$, and $z\equiv 0$, which is impossible. Hence $z(x_2)<0$. In this case we have $(F_8^1)'(T(x_2))z(x_2)>0$. Hence, $z''(x_2)<0$ which contradicts the fact that $x_2$ is a minimum.
\end{proof}

We now return to the proof of the proposition. Let $v(x)$ be a solution of problem \eqref{Eq v}.
Let us recall that $v(0)>0$. Consider, first, the case where $v(x) > 0$ for all $x \geq 0$.
We compare this solution with the solution $u(x)$ of problem \eqref{Eq v}.
Set $u_k(x) = k v(x) - u(x)$, where $k$ is the minimal positive number such that $u_k(x) \geq 0$ for
$0 \leq x \leq x_0$, where $x_0$ is the same as in Lemma 4.5. Then $u_k(x_*)=0$ for some $x_* \in (0,x_0]$
and $u_k(x) \geq 0$ for $0 \leq x \leq x_0$. If $x_* \in (0,x_0)$, then by virtue of the maximum principle,
$u_k(x) \equiv 0$. Hence $k v(x) \equiv u(x)$. We obtain a contradiction with the boundary conditions
since $v(0) > 0$ and $u(0)=0$. If $x_* = x_0$, then it follows from Lemma 4.5 that $u_k(x) \geq 0$ for all $x \geq x_0$.
Therefore, $u_k(x) \geq 0$ for all $x \geq 0$ and $u_k(x_*)=0$. As above, we obtain a contradiction with the maximum principle.

Consider now the case where $v(x)$ changes sign. Therefore, there is a value $\hat x > 0$ such that $v(\hat x) = 0$.
Since $v(x) \not\equiv 0$ for $x \geq \hat x$, then, without loss of generality, we can assume that it has some
positive values. Indeed, otherwise we multiply this solution by $-1$ since we do not use condition $v(0)>0$ anymore.
As before, we compare solutions $u(x)$ and $v(x)$ on the half-axis $x \geq \hat x$ taking into account that
$u(x) > 0$ for $x \geq \hat x$. A similar contradiction completes the proof of the proposition.

%
\end{proof}

We can now complete the proof of the existence of solutions.
Propositions \ref{prop ex pulse} and  \ref{prop line eig} affirm that problem \eqref{Eq stat tau}, \eqref{Eq stat tau boundaries} with condition \eqref{Eq stat tau mono} possesses a unique solution for $\tau=1$ and that the operator linearized about this solution does not have a zero eigenvalue. Then the topological degree for this problem is given by the expression \cite{TravelingWave}:
\begin{equation}
\gamma^*=\sum_{k=1}^{K}(-1)^{\nu_k},
\end{equation}
where $K$ is the number of solutions and $\nu_k$ is the number of positive eigenvalues for the problem linearized around each solution together with their multiplicity. Here $K=1$ since we have exactly one solution, hence, it follows that $\gamma^*\ne 0$.

 From Theorem \ref{thm bound weighed} it follows that there exists a ball $B$ containing all the monotone solutions of the equation $A^{\tau}(\mathbf{w})=\mathbf{0}$. The operator $A^{\tau}$ is proper on closed bounded set with respect to both $\mathbf{w}$ and $\tau$, hence the set of monotone solutions of this equation is compact. The separation Theorem \ref{thm separation} implies that we can construct a compact domain $D$ that contains all monotone solutions of the equation $A^{\tau}(\mathbf{w})=\mathbf{0}$ for all $\tau \in [0,1]$, and it does not contain any non monotone solution. In fact, we can consider $D$ as the union of the ball of radius $r$ around each solution. For this domain the topological degree of the problem remains unchanged along the homotopy:
\begin{equation}
\gamma(A^{\tau},D)=\gamma(A^1,D)=\gamma^*\ne 0, \; \forall \tau \in [0,1].
\end{equation}
Consequently $\gamma(A^0,D)\ne 0$, and the problem \eqref{Eq stat} possesses a solution.


\subsection{Non existence of pulses for $c\leq 0$}

We will now prove that  problem \eqref{Eq stat} does not possess any solution if the wave speed $c$ in  problem \eqref{Eq wave} is non positive. We use a similar approach as presented in \cite{marion}.
Let us first assume that the wave speed $c$ in problem \eqref{Eq wave} is negative.
Suppose that there exists a solution $\mathbf{w}(x)$ of problem \eqref{Eq stat}, and we extend this solution on $\R$ by parity.
 Let us consider the parabolic problem:

\begin{equation}\label{Eq para v}
\begin{cases}
\frac{\partial \mathbf{v}}{\partial t}=D\frac{\partial^2 \mathbf{v}}{\partial x^2}+F(\mathbf{v}),\\
\mathbf{v}(0,x)=\mathbf{v}_0(x),
\end{cases}
\end{equation}
where the initial condition satisfies the inequality:
\begin{equation}\label{cond para v ini}
\ \mathbf{v}_0(x)>\mathbf{w}(x)\;\; \forall x \in \R,
\end{equation}
and the conditions at infinity:
\begin{equation}\label{cond para v bound}
\mathbf{v}(-\infty)=\mathbf{w}^-, \ \mathbf{v}(+\infty)=\mathbf{w}^+.
\end{equation}
From the comparison theorem it follows that the
solution $v(x,t)$ of this problem satisfies the inequality $\mathbf{v}(x,t)>\mathbf{w}(x)$ for all $x$ and $t\geq 0$. Furthermore, according to \cite{TravelingWave}, the solution $\mathbf{v}(x,t)$ converges to the traveling wave $\mathbf{u}(x-ct+h)$:
\begin{equation}
\lim_{t\rightarrow +\infty} \sup_{x\in \R} |\mathbf{v}(x,t)-\mathbf{u}(x-ct+h)|=0.
\end{equation}
Here $h$ is some number.

Since $c<0$, then $\lim_{t\rightarrow +\infty} \mathbf{v}(x,t)=\mathbf{w}^+$ for any $x$. We obtain a contradiction with the inequality $\mathbf{v}(x,t)>\mathbf{w}(x)$  since $\mathbf{w}(x)>\mathbf{w}^+$.
Thus, there is no pulse solution for $c<0$.

Let us now consider the case $c=0$. Problem \eqref{Eq wave} with $c=0$ has a solution $\mathbf{u}(x)$, and for any constant $h$, the function $\mathbf{u}^h(x)=\mathbf{u}(x-h)$ is also a solution.

First, let us show that if $\mathbf{u}^h(N)>\mathbf{w}(N)$ for some $N$ large enough, then

\begin{equation}\label{qq}
  \mathbf{u}^h(x)>\mathbf{w}(x) \;\; {\rm  for} \;\; x\geq N .
\end{equation}
Indeed, the difference $\mathbf z(x) = \mathbf{u}^h(x)>\mathbf{w}(x)$
satisfies the equation similar to equation (\ref{zzz}). The required property was proved above for this equation.

Denote by $h_0$ the minimal value of $h$ for which $\mathbf{u}^h(x)>\mathbf{w}(x)$ for all $x\leq N$.
Then $\mathbf{u}^{h_0}(x) \geq \mathbf{w}(x)$ for $x\leq N$, and ${u}_i^{h_0}(x_0) = {w}_i(x_0)$ for some $x_0\leq N$
and for some component $i$ of the solution. By virtue of \eqref{qq}, $\mathbf{u}^{h_0}(x) \geq \mathbf{w}(x)$
for all $x \in \mathbb R$. Equality ${u}_i^{h_0}(x_0) = {w}_i(x_0)$ gives a contradiction with the maximum principle.
This contradiction proves that there are no pulse solutions for $c=0$.

\section{Discussion}

Blood coagulation propagates in plasma as a reaction-diffusion wave. Its properties and the convergence to this wave
from some initial distribution have important physiological meaning. We will discuss here the biological interpretation
of the mathematical results.

Let us begin with the number and stability of stationary points. As it was indicated in Section 1, they are determined by the
number and stability of zeros of the polynomial $P(T)$. Depending on parameters, it can have from one to four non negative
zeros including $T=0$ and, possibly, from one to three positive zeros. If there is only the trivial equilibrium $T=0$,
then the reaction-diffusion wave does not exist, and blood coagulation cannot occur. If there is one or three positive
zeros of the polynomial $P(T)$, then this is so-called monostable case where the point $T=0$ is unstable with respect to the corresponding
kinetic system (without diffusion). In this case, blood coagulation  begins under any small perturbation of the trivial
equilibrium which is not possible biologically. The only realistic from the physiological point of view situation
is realized with two positive equilibria corresponding to the bistable case. The trivial equilibrium $T=0$ is stable,
and blood coagulation begins if the initial perturbation exceeds some threshold level.

The initial production of blood factors after injury occurs due to the interaction of blood plasma with the damaged
vessel wall. This initial quantity of blood factors should be sufficiently large in order to overcome the threshold
level and to initiate blood coagulation. Therefore, it is important to determine these critical conditions.
We show in this work that the threshold level of the initial quantity of blood factors is given by the pulse solution
of the reaction-diffusion system. This solution exists if and only if the speed of the travelling wave is positive.
Hence, we come to the following conditions of blood coagulation: the wave speed is positive and the initial condition should
be greater than the pulse solution.

An approximate analytical method to determine the wave speed is suggested in \cite{BOUCHNITA201674}.
The system of equation is reduced to a single equation by some asymptotic procedure. It is the same equation for the
thrombin concentration
\begin{equation}\label{ss}
  \frac{\partial T}{\partial t} = D \; \frac{\partial^2 T}{\partial x^2} + P(T)
\end{equation}
as the equation considered above to construct the homotopy in the Leray-Schauder method.
There is a simple analytical condition providing the positiveness of the wave speed for this equation:
the integral $\int_0^{T^-} P(T)dT$ should be positive. If the wave speed is negative, then blood coagulation does not occur.
If the wave speed is positive but it is less than some given physiological values, then blood coagulation is insufficient
leading to possible bleeding disorders such as hemophilia. If the speed is too large, then excessive blood coagulation
can lead to thrombosis. The analytical approach suggested in \cite{BOUCHNITA201674} gives a good approximation of the
wave speed obtained for the system of equations and in the experiments.

The second criterium of blood coagulation concerns the initial condition. If the wave speed is positive, then, as we prove in this
work, there exists a pulse solution. The initial condition should be greater than this pulse solution. For the scalar equation
the solution will then locally converge to the second stable equilibrium. It is also proved that it converges to the travelling
wave solution. For the system of equations, local convergence to the second stable equilibrium can also be proved using
the method of upper and lower solutions. Convergence to the travelling wave is not proved but it can be expected
since it takes place for the scalar equation approximating the system of equations.
If the initial condition is less than the pulse solution, then the solution will decay converging to zero. Blood coagulation
does not occur in this case. Mathematically, this convergence can be proved both for the scalar equation and for the
system of equations by the same method of upper and lower solutions.

The results of this work essentially use the monotonicity property of the reaction-diffusion system. This property implies
the applicability of the maximum principle and of some other mathematical methods. The model considered here is a simplification
of more complete models of blood coagulation. Though they may not satisfy the monotonicity property, we can expect that
similar qualitative properties of solutions will also be valid for these more complete models.

Finally, let us note that the approach developed in this work can be applied for some models of blood coagulation in flow.
This question will be investigated in the subsequent works.



\bigskip

\noindent
\paragraph{Acknowledgements.}

The work was partially supported by the ``RUDN University Program 5-100''.



\newpage

\bibliographystyle{apalike}
\bibliography{biblio}

\appendix

\section{Expression of the function $\boldsymbol{\phi}$}
\label{App Exp phi}

The function $\boldsymbol{\phi}(T) = (\phi_1(T),...,\phi_7(T))$ such that $\mathbf F(\boldsymbol{\phi}(T),T)=0$  is given by the following expressions:
\begin{align}
v_1= & \frac{k_1}{h_1}\phi_3(T)\phi_6(T)=\phi_1(T),\label{Equation phi1}\\
v_2= & \frac{k_2}{h_2}\phi_4(T)\phi_5(T)=\phi_2(T),
\label{Equation phi2}\\
v_3= & \frac{k_3\rho_3T}{k_3T+h_3}=\phi_3(T), \label{Equation phi3} \\
v_4= & \frac{k_4\rho_4T}{k_4T+h_4}=\phi_4(T),
\label{Equation phi4}\\
v_5= & \frac{k_5\rho_5\phi_7(T)}{k_5\phi_7(T) +h_5}=\phi_5(T),
\label{Equation phi5}\\
v_6= & \frac{\rho_6(k_6\phi_5(T)+\bar{k_6}\phi_2(T))}{k_6\phi_5(T)+\bar{k_6}\phi_2(T)+h_6}=\phi_6(T),
\label{Equation phi6} \\
v_7= & \frac{k_7\rho_7T}{k_7T+h_7}=\phi_7(T),
\label{Equation phi7}
\end{align}

For each $i=1,...,7$, the function $\phi_i$ satisfies $F_i(\phi_i(T),T)=0$ for $T\geq 0$. Then $\boldsymbol{\phi}$ is obtain by successively determining: $\phi_3$, $\phi_4$, $\phi_7$, $\phi_5$, $\phi_2$, $\phi_6$ and $\phi_1$. This order of computing will often be used in this work.  It can be noticed that the functions $\phi_i$ are rational fractions with no positive poles.

The following properties of the functions $\phi_i$ are straightforward:
\begin{equation}
\begin{cases}
\phi_i(0)=0, \\
\phi_i(T)>0 \ \text{and } \phi_i'(T)  >0 \text{ for } T>0, \\
\underset{T\rightarrow +\infty}{\text{lim }}\phi_i(T)=w_{i,\ell} <\infty.
 \end{cases}
 \label{Properties phi}
\end{equation}
 Furthermore we note that for $i=3,...,7$ and $T<T^0$, $\phi_i(T)<\rho_i$.


 \bigskip

\section{Description of $P$}
\label{App Poly P}

In the Appendix \ref{App Exp phi} we saw that the $\phi_i$ are rational fractions, hence it is also the case for $P$ since $P$ satisfies \eqref{Equation P}. We have:
\begin{equation}
P(T)=\frac{P_1(T)(T^0-T)}{P_2(T)}-h_8T,
\end{equation}
where $P_1$ and $P_2$ are third degree polynomials with positive coefficients, hence $P_2(T)>0$ for $T\geq 0$. Then the numerator of $P$ denoted $R$ is:
\begin{equation}
R(T)=P_1(T)(T^0-T)-h_8P_2(T)T.
\label{Polynom R}
\end{equation}
The computation of the coefficients of $R$ shows that it is a fourth degree polynomial: $R(T)=aT^4+bT^3+cT^2+dT$.
Since the coefficient of $P_1$ and $P_2$ are positive, we can note that the coefficient $a$ is negative.
Only the explicit form of the coefficient $d$ is used:
\begin{equation}
d=k_5\rho_5k_6\rho_6k_7\rho_7k_8\rho_8h_1h_2h_3h_4 - h_1h_2h_3h_4h_5h_6h_7h_8
\label{Coef d}
\end{equation}


\bigskip

\section{Jacobian matrix $(\mathbf {F}^{\tau})'$}
\label{App Jac Mat}

The matrix has the form:
\begin{equation}
\mathcal{J}^{\tau}(v) =
\begin{pmatrix}
-d_1&0&\theta_{1,3}&0&0&\theta_{1,6}&0&0\\
0&-d_2&0&\theta_{2,4}&\theta_{2,5}&0&0&0\\
0&0&-d_3&0&0&0&0&\theta_{3,8}\\
0&0&0&-d_4&0&0&0&\theta_{4,8}\\
0&0&0&0&-d_5&0&\theta_{5,7}&0\\
0&\theta_{6,2}&0&0&\theta_{6,5}&-d_6&0&0\\
0&0&0&0&0&0&-d_7&\theta_{7,8}\\
\theta_{8,1}&0&0&0&0&\theta_{8,6}&0&H^{\tau}
\end{pmatrix}
\label{Jacobian matrix}
\end{equation}

\vspace{5pt}
where $H^\tau = P'(T)$,


\begin{equation}
\begin{cases}
d_1=h_1 , \;\;
d_2=h_2 , \;\;
d_3=k_3T+h_3 , \;\;
d_4=k_4T+h_4 , \;\; \\
d_5=k_5v_7+h_5 , \;\;
d_6=\overline{k_6}v_2+k_6v_5+h_6 , \;\;
d_7= k_7T+h_7
\end{cases}
\end{equation}

\begin{equation}
\begin{cases}
\theta_{1,3}=k_1v_6 , \;\;
\theta_{1,6}=k_1v_3 , \;\;
\theta_{2,4}=k_2v_5 , \\
\theta_{2,5}=k_2v_4 , \;\;
\theta_{3,8}=k_3(\rho_3-v_3)
\end{cases}
\end{equation}

 \begin{equation}
\begin{cases}
\theta_{4,8}=k_4(\rho_4-v_4) , \;\;
\theta_{5,7}=k_5(\rho_5-v_5) , \\
\theta_{6,2}=\overline{k_6}(\rho_6-v_6) , \;\;
\theta_{6,5}=k_6(\rho_6-v_6)\\
\theta_{7,8}=k_7(\rho_7-v_7) , \;\;
\theta_{8,1}=\overline{k_8}\alpha^{\tau}(T^0-T) , \;\;
\theta_{8,6}=k_8\alpha^{\tau}(T^0-T)
\end{cases}
\end{equation}


\end{document}